\documentclass[a4paper,11pt]{article}
\usepackage[T2A]{fontenc}
\usepackage[utf8]{inputenc}
\usepackage[russian,english]{babel}

\usepackage{latexsym}
\usepackage{geometry}
\usepackage{graphicx}
\usepackage{amsfonts}
\usepackage{amsthm}
\usepackage{amsmath}
\usepackage{amssymb}
\usepackage{xcolor}
\usepackage{mathrsfs}
\usepackage{bookmark}
\usepackage{hyperref}

\usepackage{accents}
\usepackage{indentfirst}
\usepackage{enumerate}
\usepackage{float}
\usepackage{subcaption}
\usepackage{mathtools}
\usepackage{csquotes}
\usepackage{setspace}
\usepackage[noadjust]{cite}
\usepackage{array}

\newtheorem{theorem}{Theorem}[section]
\newtheorem{corollary}[theorem]{Corollary}
\newtheorem{proposition}[theorem]{Proposition}
\newtheorem{lemma}[theorem]{Lemma}
\theoremstyle{definition}
\newtheorem{definition}[theorem]{Definition}

\newtheorem{notation}[theorem]{Notation}
\newtheorem{example}[theorem]{Example}
\newtheorem{remark}[theorem]{Remark}

\numberwithin{equation}{section}
\allowdisplaybreaks

\let\phi\varphi

\def\A{\mathcal{A}}

\def\Main{\mathcal{M}}

\def\Zorn{{\rm Zorn}}

\renewcommand{\Im}{\mathop{\mathfrak{Im}}\nolimits}

\DeclareMathOperator{\chrs}{char}

\DeclareMathOperator{\spn}{span}
\DeclareMathOperator{\Anc}{Anc}

\newcommand\til[1]{\widetilde{\phantom{,\mkern-4mu} #1 \phantom{,\mkern-4mu}}\mkern-1mu}

\providecommand{\keywords}[1]{\textbf{Keywords:} #1}
\providecommand{\msc}[1]{\textbf{MSC 2020:} #1}

\newcommand\freefootnote[1]{%
    \bgroup
    \renewcommand\thefootnote{\fnsymbol{footnote}}%
    \renewcommand\thempfootnote{\fnsymbol{mpfootnote}}%
    \footnotetext[0]{#1}%
    \egroup
}

\begin{document}

\title{Diameter of the commutativity graph of the real sedenions}
\author{
Svetlana Zhilina$^{a}$
}
\date{\small \em
$^a$Department of Mathematics and Mechanics,\\ Lomonosov Moscow State University,\\ Moscow, 119991, Russia
}

\maketitle

\begin{abstract}
The commutativity graph of the real sedenion algebra is considered. It is shown that those elements whose imaginary part is not a zero divisor correspond to isolated vertices of this graph. All other elements form a connected component whose diameter equals $3$.
\end{abstract}


\msc{05C25, 17A20, 17D05}

\freefootnote{This work was supported by by the Russian Science Foundation (project No.~22-11-00052).}

\freefootnote{Email address: \texttt{s.a.zhilina@gmail.com}}

\section{Introduction}

One of the ways to visualize an arbitrary binary algebraic relation $R$ is to define the corresponding graph. Its vertices represent elements or their equivalence classes in an algebraic structure under consideration, and there is an edge from $x$ to $y$ if and only if $xRy$. The most popular relation graphs of various rings and algebras are commutativity, orthogonality, and zero divisor graphs.

Studying relation graphs appears to be particularly interesting in the case of nonassociative algebras. Cayley--Dickson algebras~$\A_n$, $n \in \mathbb{N}_0$, over an arbitrary field~$\mathbb{F}$, $\chrs \mathbb{F} \neq 2$, form an important class of such algebras. They are a family of $2^n$-dimensional algebras which are defined inductively: $\A_0 = \mathbb{F}$, and at each step the algebra $\A_{n+1}$ is obtained from~$\A_n$ by applying the Cayley--Dickson construction with some nonzero parameter $\gamma_n \in \mathbb{F}$. The algebra $\A_{n+1}$ consists of ordered pairs of elements from~$\A_n$, i.e., its elements have the form $(a,b) \in \A_n \times \A_n$. For $n = 3$ the algebra~$\A_n$ becomes nonassociative, and for $n \geq 4$ it is not even alternative, that is, the identities $a(ab) = (aa)b$ and $(ba)a = b(aa)$ are not satisfied in~$\A_n$.

An important particular case of Cayley--Dickson algebras are the so called real algebras of the main sequence, which we denote by~$\Main_n$. In these algebras we have $\mathbb{F} = \mathbb{R}$, and all Cayley--Dickson parameters are equal to~$-1$. Every Cayley--Dickson algebra has a norm defined in a natural way, and in the case of real algebras of the main sequence this norm is anisotropic, i.e., $n(a) \neq 0$ for all $a \in \Main_n \setminus \{ 0 \}$. It then easily follows that~$\Main_n$ cannot have zero divisors whenever it is alternative. Therefore, the first algebra of the main sequence with zero divisors is the sedenion algebra $\mathbb{S} = \Main_4$.

Among the authors who studied zero divisors in real algebras of the main sequence, we mention Moreno~\cite{moreno, moreno_alternative, moreno_constructing} and Biss, Dugger, and Isaksen~\cite{biss, biss2}. Particularly, in~\cite{biss, biss2} the dimensions of their annihilators were completely described, and the zero divisors whose annihilators have the largest possible dimension were classified. Then Pixton~\cite{pixton} obtained a similar result on the dimension of alternators in these algebras. It should be noted that Moreno was the first to study doubly alternative elements in real algebras of the main sequence, that is, the elements whose both components are alternative in the previous algebra of the sequence. He established several important properties of doubly alternative zero divisors, see~\cite[pp.~25--27]{moreno}.

Properties of relation graphs of Cayley--Dickson algebras have been studied in the author's papers~\cite{our_orthographs1,our_split-sedenions,our_split-algebras,our_doubly-alternative,our_sedenions}. In~\cite{our_orthographs1} zero divisors whose components satisfy additional conditions on their norm and alternativity have been studied in an arbitrary real Cayley--Dickson algebra. It has been shown that they form hexagonal patterns in orthogonality and zero divisor graphs, see~\cite[Corollary~3.7]{our_orthographs1}. In the algebras of the main sequence, each hexagon can be extended to the so-called double hexagon, see~\cite[Description~4.8]{our_orthographs1}, and the multiplication table of its vertices has a block structure, cf.~\cite[Theorem~4.11]{our_orthographs1}.

In~\cite{our_doubly-alternative} these results have been extended to the case of arbitrary Cayley--Dickson algebras over a field~$\mathbb{F}$, $\chrs \mathbb{F} \neq 2$, and in the general case algebras of the main sequence correspond to Cayley--Dickson algebras with anisotropic norm. Besides, most of the results of the papers~\cite{moreno,biss} have also been generalized to Cayley--Dickson algebras with anisotropic norm. So, according to~\cite[Theorem~6.11]{our_doubly-alternative}, if $n \geq 2$ and~$\A_n$ is a Cayley--Dickson algebra with anisotropic norm, then the dimension of an annihilator of any element from~$\A_n$ is divisible by four. The proof of this and several other facts is based on the properties of operators of left and right multiplication by a fixed element. In particular, in a Cayley--Dickson algebra~$\A_n$ with anisotropic norm, left multiplication by an element whose both components have zero trace corresponds to a conjugate-linear mapping which is skew-Hermitian with respect to a certain Hermitian inner product on~$\A_n$. It then follows that the codimension (and hence the dimension also) of its kernel is divisible by four.

Another important subclass of Cayley--Dickson algebras are Cayley--Dickson split-algebras. They have isotropic norm, so there appear zero divisors even in the case when the algebra~$\A_n$ is alternative. It is well known (\cite[pp.~160 and~166]{mccrimmon}) that the split-quaternions $\A_2$ over an arbitrary field~$\mathbb{F}$, $\chrs \mathbb{F} \neq 2$, are isomorphic to the algebra $M_2(\mathbb{F})$ of square matrices of order~$2$ over~$\mathbb{F}$, and the split-octonions $\A_3$ are isomorphic to Zorn vector-matrix algebra whose definition is given in Subsection~\ref{subsection:octonions-sedenions}. With the help of these isomorphisms, operating with the elements of these algebras can be simplified significantly. In particular, Lopatin and Zubkov used them to find the orbits of all elements and of all pairs of elements, in the split-octonions over an infinite field~$\mathbb{F}$, with respect to the action of the automorphism group, see~\cite[Proposition~3.4 and Theorem~4.1]{lopatin}. In~\cite{our_split-algebras,our_split-sedenions} commutativity, orthogonality, and zero divisor graphs of low-dimensional real Cayley--Dickson split-algebras (with $n \leq 4$) are described in detail.

The aim of the current paper is to study connected components of the commutativity graph of the real sedenions~$\mathbb{S}$, denoted by $\Gamma_C(\mathbb{S})$. Thus we continue the research of relation graphs of the sedenion algebra which began in~\cite{our_sedenions}. One of the main results of the previous paper is the complete description of connected components of the orthogonality graph $\Gamma_O(\mathbb{S})$. They can be mapped bijectively to lines in the imaginary part of the octonions, cf.~\cite[Theorem~4.15]{our_sedenions}. For each of the connected components, its vertex set is described explicitly, and its diameter is proved to be equal to three, see~\cite[Corollary~4.10 and Theorem~4.11]{our_sedenions}.

The structure of this paper is as follows: In Section~\ref{section:definitions} we introduce main definitions and notations which are used throughout the text. In particular, we describe the Cayley--Dickson process in detail in Subsection~\ref{subsection:A_n} and mention some of the properties of Cayley--Dickson algebras in Subsection~\ref{subsection:A_n-properties}. In Subsection~\ref{subsection:main-sequence} we recall main results of the paper~\cite{our_doubly-alternative} on orthogonality graphs of Cayley--Dickson algebras with anisotropic norm. Then in Subsection~\ref{subsection:octonions-sedenions} we give explicit definitions of real octonions and sedenions and consider matrix representations of octonions.

Finally, in Section~\ref{section:commutativity-graph} we consider the commutativity graph of~$\mathbb{S}$. The key role is played by Lemma~\ref{lemma:A_n-commutativity-through-orthogonality} which describes a relationship between the centralizer and the orthogonalizer of an arbitrary element. It follows immediately that those elements whose imaginary part is not a zero divisor correspond to isolated vertices in~$\Gamma_C(\mathbb{S})$. All other elements form the subgraph~$\Gamma_C^Z(\mathbb{S})$, and Theorem~\ref{theorem:diameter-3} shows that it is a connected component whose diameter equals $3$. The proof uses Lemma~\ref{lemma:length-two-path} 
which provides a criterion for the distance between two arbitrary elements in $\Gamma_C^Z(\mathbb{S})$ to be at most two. Note that Theorem~\ref{theorem:diameter-3} contains an explicit algorithm for constructing a path of length at most three between an arbitrary pair of elements, which can be reduced to solving a system of four linear equations in five variables.

\section{Main definitions and notations} \label{section:definitions}

\subsection{Algebraic relations and their graphs} \label{subsection:definitions}

Let~$\mathbb{F}$ be an arbitrary field and $(\A, +, \cdot)$ be an algebra over~$\mathbb{F}$, possibly noncommutative and nonassociative. We say that $a, b \in \A$ {\em anticommute} if $ab + ba = 0$, and $a, b \in \A$ are {\em orthogonal} if $ab = ba = 0$. We denote the set of zero divisors (left, right, or two-sided) in~$\A$ by $Z(\A)$, the set of two-sided zero divisors in~$\A$ by $Z_{LR}(\A)$, and the (commutative) center of~$\A$ by $C_{\A}$.

\begin{definition} \label{definition:subspaces}
Let $a$ be an arbitrary element of~$\A$.
\begin{itemize}
    \item
    {\em The centralizer} of $a$ is $C_\A(a) = \big\{ b \in \A \: | \: ab=ba \big\}$, i.e., the set of all elements in~$\A$ which commute with $a$.
    \item
    {\em The anticentralizer} of $a$ is $\Anc_\A(a) = \big\{ b \in \A \: | \:  ab+ba=0 \big\}$, i.e., the set of all elements in~$\A$ which anticommute with $a$.
    \item
    {\em The orthogonalizer} of $a$ is $O_\A(a)=\big\{ b \in \A \: | \;  ab=ba=0 \big\}$, i.e., the set of all elements in~$\A$ which are orthogonal to $a$.
\end{itemize}
\end{definition}

It is clear that $C_\A(a)$, $\Anc_\A(a)$, and $O_\A(a)$ are linear spaces over~$\mathbb{F}$.

\begin{notation}
For any subset $X$ of a linear space $W$ over~$\mathbb{F}$ we denote the set of lines passing through nonzero elements of~$X$ by
$$
\mathbb{P}(X) = \{ [x] = \mathbb{F} x \; | \; x \in X \setminus \{ 0 \} \}.
$$
\end{notation}

We now introduce some relation graphs which are to be studied in this paper.

\begin{definition} \label{definition:graphs}
Let~$\A$ be an arbitrary algebra. We define the following relation graphs of~$\A$:
\begin{itemize}
\item 
{\em The commutativity graph} $\Gamma_C(\A)$: its vertices are elements of 
$$
\mathbb{P}(\A/C_\A) = \{ [a + C_{\A}] = \mathbb{F}a + C_{\A} \; | \; a \in \A \setminus C_\A \},
$$
and distinct vertices $[a + C_{\A}]$ and $[b + C_{\A}]$ are adjacent if and only if $ab = ba$.
\item 
{\em The orthogonality graph} $\Gamma_O(\A)$: its vertices are elements of $\mathbb{P}(Z_{LR}(\A))$, and distinct vertices $[a]$ and $[b]$ are adjacent if and only if $ab = ba = 0$.
\end{itemize}
\end{definition}

Note that the edges of $\Gamma_C(\A)$ and $\Gamma_O(\A)$ are well-defined. When speaking of the vertices of these graphs, we will not distinguish between a nonzero element $a$ and a line $[a] = \mathbb{F} a$ passing through it. We also denote $\spn(a_1, \dots, a_k) = \mathbb{F} a_1 + \dots + \mathbb{F} a_k$.

Recall that, in an undirected graph $\Gamma$, $d(x,y) = d_{\Gamma}(x,y)$ denotes {\em the distance} between two vertices $x$ and $y$, and $d(\Gamma) = \sup\limits_{x,y \in \Gamma} d(x,y)$ denotes {\em the diameter} of $\Gamma$.

\subsection{Constructing Cayley--Dickson algebras} \label{subsection:A_n}

We refer the reader to~\cite{mccrimmon,schafer} for auxiliary definitions and general properties of Cayley--Dickson algebras.

\begin{definition} \label{definition:cayley-dickson-algebras}
Let~$\A$ be an algebra over a field~$\mathbb{F}$ with an involution $a \mapsto \bar{a}$. The algebra $\A \{ \gamma \}$ produced by {\em the Cayley--Dickson process}, when applied to~$\A$ with the parameter $\gamma \in \mathbb{F}$, $\gamma \neq 0$, is defined as the set of ordered pairs of elements of~$\A$ with operations
\begin{align*}
\alpha(a,b)&=(\alpha a, \alpha b);\\
(a,b)+(c,d)&=(a+c,b+d);\\
(a,b)(c,d)&=(ac+\gamma \bar{d}b,da+b\bar{c})
\end{align*} 
and the involution
$$
\qquad (\overline{a,b})=(\bar{a},-b), \qquad a,b,c,d\in \A, \ \alpha \in \mathbb{F}.
$$
\end{definition}

\begin{remark}[{\cite[p.~435]{schafer}}]
If the algebra~$\A$ is unital, with $1_{\A}$ being its unity, then $(1_{\A}, 0)$ is the unity of $\A \{ \gamma \}$. If, moreover, the involution on~$\A$ is regular, that is, $a + \bar{a} \in \mathbb{F}1_{\A}$ and $a\bar{a} = \bar{a}a \in \mathbb{F}1_{\A}$ for all $a \in \A$, then the involution on $\A \{ \gamma \}$ is also regular.
\end{remark}

Henceforth we assume that $\chrs \mathbb{F} \neq 2$. We now define an arbitrary Cayley--Dickson algebra which is determined by the set of its parameters.

\begin{definition} \label{definition:A_n}
For every integer $n \geq 0$ and nonzero numbers $\gamma_0, \dots, \gamma_{n-1} \in \mathbb{F}$ we define the Cayley--Dickson algebra $\A_n = \A_n \{ \gamma_0, \dots, \gamma_{n-1} \}$ inductively:
\begin{enumerate} [(1)]
    \item $\A_0 = \mathbb{F}$;
    \item $\A_{n+1} \{ \gamma_0, \dots, \gamma_n \}=(\A_n \{ \gamma_0, \dots, \gamma_{n-1} \}) \{ \gamma_n \}$.
\end{enumerate}
\end{definition}

For every integer $n \geq 0$ the structure~$\A_n$ in Definition~\ref{definition:A_n} is a unital $2^n$-dimensional algebra over~$\mathbb{F}$ with a regular involution.

\begin{definition} \label{definition:real-imaginary-part}
Let $a \in \A_n$. Its {\em trace} is $t(a) = a + \bar{a}$, its {\em imaginary part} is $\Im(a) = \frac{a - \bar{a}}{2}$, and its {\em norm} is $n(a) = a \bar{a} = \bar{a}a$. Since the involution on~$\A_n$ is regular, we have $t(a), n(a) \in \mathbb{F}$.
\end{definition}

\begin{proposition}[{\cite[p.~435]{schafer}}] \label{proposition:real-cayley-dickson-properties}
We can compute trace and norm of an element $(a,b) \in \A_{n+1}$ inductively by using the following equalities:
\begin{align*}
	t((a,b)) &= t(a),\\
	n((a,b)) &= n(a) - \gamma_n n(b).
\end{align*}
\end{proposition}

It follows from Proposition~\ref{proposition:real-cayley-dickson-properties} that the norm $n(\cdot)$ is a nondegenerate quadratic form on~$\A_n$.

\subsection{Some properties of Cayley--Dickson algebras} \label{subsection:A_n-properties}

Henceforth we assume that~$\A$ is an arbitrary algebra over a field~$\mathbb{F}$, and $\A_n = \A_n \{ \gamma_0, \dots, \gamma_{n-1} \}$ is an arbitrary Cayley--Dickson algebra over a field~$\mathbb{F}$, $\chrs \mathbb{F} \neq 2$.

\begin{proposition}[{\cite[p.~440]{schafer}}] \label{proposition:lambda-form}
Let $\langle a, b \rangle$ denote an~$\mathbb{F}$-valued symmetric bilinear form on~$\A_n$ associated with the quadratic form $n(a)$. Then $\langle a, a \rangle = n(a)$ and $2\langle a,b \rangle = a \bar{b} + b \bar{a} = \bar{a} b + \bar{b} a = t(a\bar{b})$ for all $a, b \in \A_n$. Besides, for any $a, b \in \A_n$ it holds that $\langle a, b \rangle = \langle \bar{a}, \bar{b} \rangle$ and $t(a) = 2 \langle a, e_0 \rangle$.
\end{proposition}

\begin{notation}
We denote $a \perp b$ if $a$ and $b$ are orthogonal with respect to $\langle \cdot, \cdot \rangle$, that is, $\langle a, b \rangle = 0$.
\end{notation}

The following lemma describes the anticentralizer of an arbitrary nonzero element in~$\A_n$ with zero trace.

\begin{lemma}[{\cite[Lemma~5.8]{our_anticomm}}] \label{lemma:A_n-anticomm}
Let $a \in \A_n \setminus \{ 0 \}$ and $t(a) = 0$. Then
\begin{equation*}
    \Anc_{\A_n}(a) = \left\{ b \in \A_n \; | \; t(b) = 0 \mbox{ and } \langle a,b \rangle = 0 \right\}.
\end{equation*}
\end{lemma}

We now proceed to some concepts related to associativity. For $a,b,c \in \A$ we denote their associator by $[a,b,c] = (ab)c - a(bc)$, and their anti-associator by $\{ a,b,c \} = (ab)c + a(bc)$. An algebra~$\A$ is called {\em alternative} if for all $a,b \in \A$ the equalities $[a,a,b] = [b,a,a] = 0$ hold, and it is called {\em flexible} if for all $a,b \in \A$ it holds that $[a,b,a] = 0$. In a flexible algebra~$\A$, we have $[a,b,c]=-[c,b,a]$ for all $a,b,c \in \A$. It is well known that~$\A_n$ is alternative if and only if $n \leq 3$, however, all Cayley--Dickson algebras are flexible, see~\cite[p.~436, Theorem~1]{schafer}.

\begin{definition}[{\cite[p.~12, p.~15]{moreno_alternative}}]
Let $a, b \in \A_n$.
\begin{itemize}
    \item We say that $a$ {\em alternates} with $b$ if $[a,a,b] = 0$.
    \item If $a$ alternates with every $b \in \A_n$, then $a$ is {\em alternative}.
    \item We say that $a$ {\em alternates strongly} with $b$ if $[a,a,b] = 0$ and $[b,b,a] = 0$.
    \item If $a$ alternates strongly with every $b \in \A_n$, then $a$ is {\em strongly alternative}.
\end{itemize}
\end{definition}

\section{Algebras with anisotropic norm}

\subsection{Zero divisors in algebras with anisotropic norm} \label{subsection:main-sequence}

In this section we consider those Cayley--Dickson algebras over an arbitrary field~$\mathbb{F}$, $\chrs \mathbb{F} \neq 2$, such that the norm form is anisotropic, i.e., $n(a) = 0$ if and only if $a = 0$. We denote $\til{e}_0 = (0,e_0) \in \A_n$ and $\til{a} = a \til{e}_0$ for all $a \in \A_n$. Note that for $a = (x,y)$ we have $\til{a} = (\gamma_n y, x)$.

According to~\cite[Corollary~4.16]{our_doubly-alternative}, in Cayley--Dickson algebras all zero divisors appear to be two-sided zero divisors, that is, $Z(\A_n) = Z_{LR}(\A_n)$. However, in case of Cayley--Dickson algebras with anisotropic norm a stronger result holds.

\begin{lemma}[{\cite[Corollaries~1.6 and~1.12]{moreno}, \cite[Lemma~5.1]{our_doubly-alternative}}]  \label{lemma:moreno-zero-divisor-conditions} \label{lemma:moreno-doubly-pure} \label{lemma:moreno-tilde}
Let~$\A_n$ be a Cayley--Dickson algebra with anisotropic norm, $a, b \in \A_n$. Then $ab = 0$ if and only if $ba = 0$ if and only if $a\til{b} = 0$.
\end{lemma}

\begin{lemma}[{\cite[pp.~25--27]{moreno}, \cite[Lemma~5.6]{our_doubly-alternative}}] \label{lemma:A_n-alternative-properties} \label{lemma:alternative-orthogonal-system}
Let $\A_{n+1}$ be a Cayley--Dickson algebra with anisotropic norm, the elements $c,d \in \A_n$ alternate with $a,b \in \A_n$, $(a,b),(c,d) \in Z(\A_{n+1})$, $(a,b)(c,d) = 0$. Then
\begin{enumerate}[{\rm (1)}]
    \item $t(a) = t(b) = t(c) = t(d) = 0$;
    \item $n(a) = -\gamma_n n(b)$ and $n(c) = -\gamma_n n(d)$;
    \item $[c,a,d] = 2n(c)b$, $[c,b,d] = -2n(d)a$;
    \item $\{ c,a,d \} = \{ c,b,d \} = 0$;
    \item $a \perp b$;
    \item $a, b \in \spn(e_0,c,d,cd)^{\perp}$;
    \item $(c,d)(ac,ad) = 0$.
\end{enumerate}
\end{lemma}

\begin{corollary}[{\cite[Proposition~11.1]{biss}, \cite[Lemma~4.19]{our_doubly-alternative}}] \label{lemma:double-alternative-annihilators}
Let $\A_{n+1}$ be a Cayley--Dickson algebra with anisotropic norm, $(c,d) \in Z(\A_{n+1})$, and the elements $c, d$ are alternative in~$\A_n$. Then
$$
O_{\A_{n+1}}((c,d)) = \left\{ \left(a, \dfrac{(ca)d}{n(c)} \right) \; \bigg| \; t(a) = 0, \: \{ c,a,d \} = 0 \right\}.
$$
\end{corollary}

\begin{proof}
According to~\cite[Lemma~4.19]{our_doubly-alternative},
\begin{equation} \label{equation:orthogonalizer}
O_{\A_{n+1}}((c,d)) = \left\{ \left(a, -\dfrac{(da)c}{n(c)} \right) \; \bigg| \; t(a) = 0, \: d(ac) = \chi (da)c \right\},
\end{equation}
where $\chi = \frac{\gamma_n n(d)}{n(c)}$. Denote $b = -\frac{(da)c}{n(c)} = -\chi \frac{d(ac)}{n(c)}$. Since $c, d$ are alternative in~$\A_n$, they, in particular, alternate with $a, b \in \A_n$. By Lemma~\ref{lemma:A_n-alternative-properties}(2), $\chi = -1$, so the last condition on the right-hand side of Eq.~\eqref{equation:orthogonalizer} takes on the form $\{ d,a,c \} = (da)c + d(ac) = 0$. It follows from Lemma~\ref{lemma:A_n-alternative-properties}(1) that $t(a) = t(b) = t(c) = t(d) = 0$. Then $\{ c,a,d \} = -\overline{\{ d,a,c \}} = 0$ and $b = -\bar{b} = -\overline{\frac{d(ac)}{n(c)}} = \frac{(ca)d}{n(c)}$.
\end{proof}

\begin{theorem}[{\cite[Theorem~5.11]{our_doubly-alternative}}]
\label{theorem:double-hexagon}
Let $\A_{n+1}$ be a Cayley--Dickson algebra with anisotropic norm, $(a,b),(c,d) \in Z(\A_{n+1})$, $(a,b)(c,d) = 0$, and the elements $a,b \in \A_n$ alternate strongly with $c,d \in \A_n$, i.e.,  $[x,x,y] = [y,y,x] = 0$ for $x \in \{ a, b \}$ and $y \in \{ c, d\}$. Then
\begin{enumerate}[{\rm (1)}]
    \item The elements $ac,ad$ alternate strongly with $a,b,c,d$. \label{item:strong-alternativity}
    \item The elements $e_0,a,b,c,d,ac,ad$ are orthogonal with respect to $\langle \cdot, \cdot \rangle$. \label{item:orthonormal-system}
    \item There exists a subgraph of $\Gamma_O(\A_{n+1})$ which is depicted in Figure~\ref{figure:double-hexagon}(a) and called a double hexagon. \label{item:double-hexagon}
    \item All elements in the vertices of the double hexagon are linearly independent. \label{item:linear-independence}
\end{enumerate}
\end{theorem}

\begin{figure}[ht]
\centering
{\includegraphics[width=0.74\linewidth]{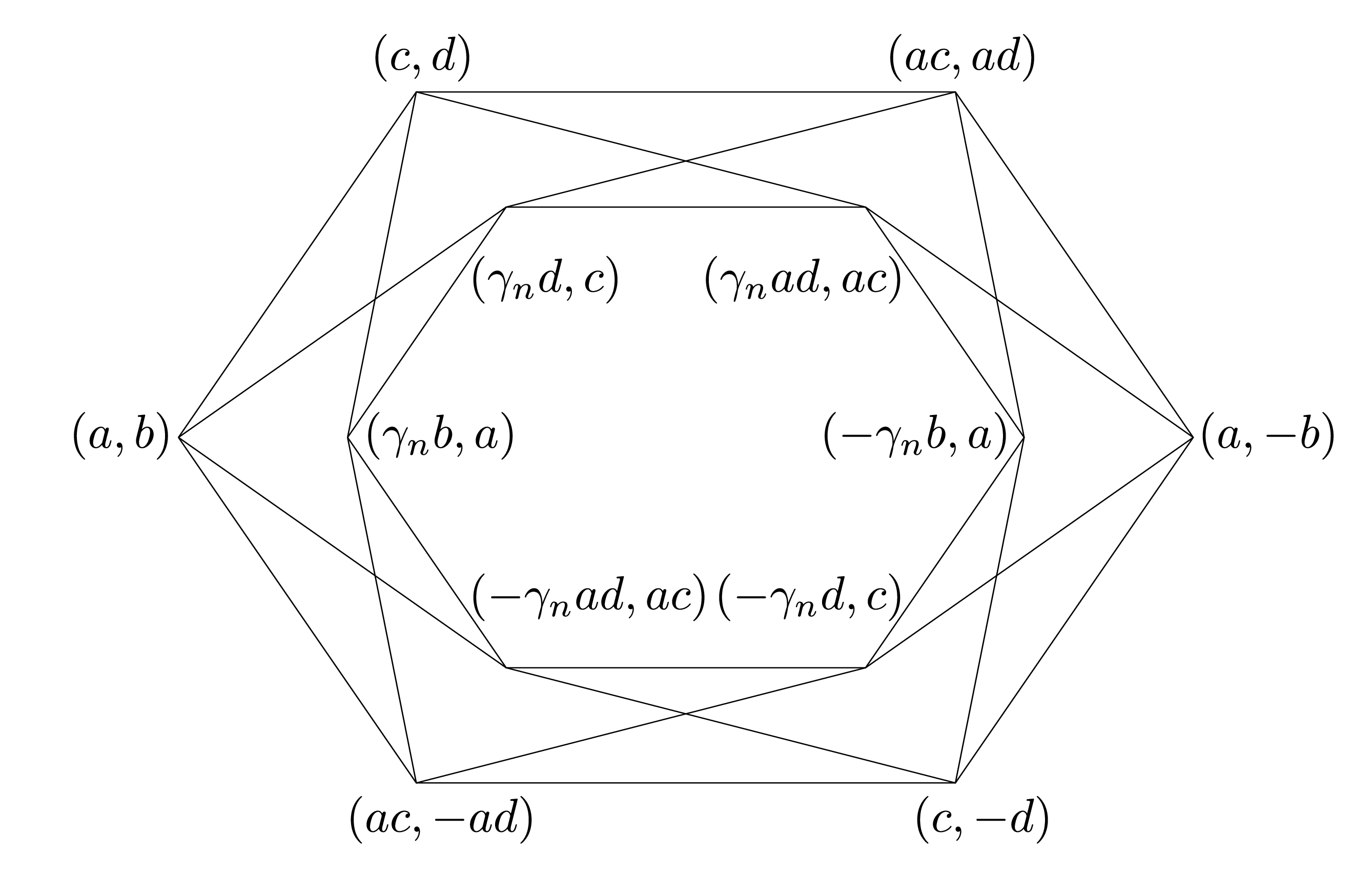}
\subcaption{Algebra with anisotropic norm.}}
{\includegraphics[width=0.55\linewidth]{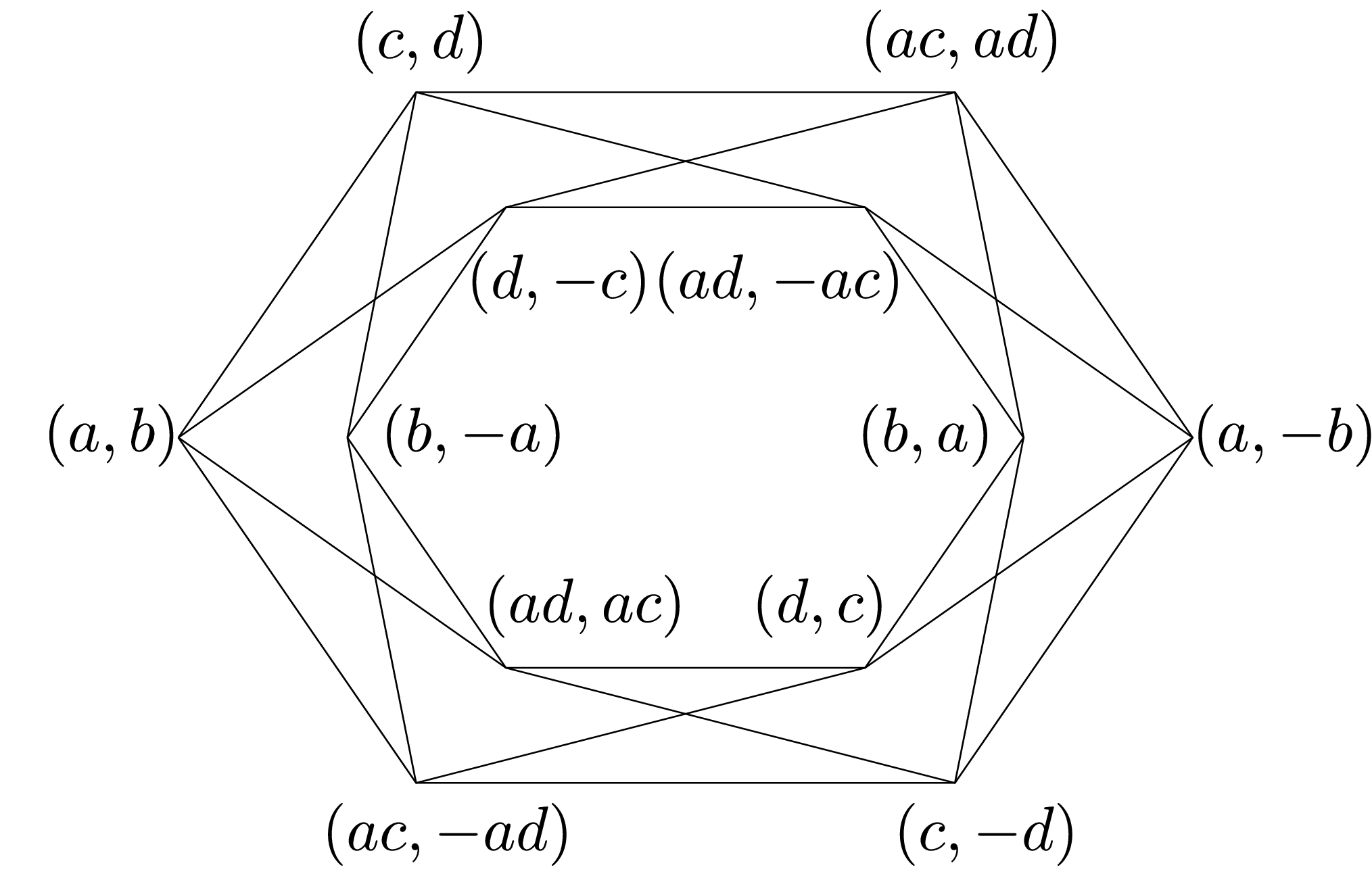}
\subcaption{Algebra of the main sequence.}}
\caption{\label{figure:double-hexagon}Double hexagons.}
\end{figure}

An important special case of Cayley--Dickson algebras with anisotropic norm are the so called real Cayley--Dickson algebras of the main sequence.

\begin{definition} \label{definition:A_n-examples}
Let $\mathbb{F} = \mathbb{R}$. The algebra $\A_n \{ \gamma_0, \dots, \gamma_{n-1} \}$ is called {\em a real Cayley--Dickson algebra of the main sequence} if $\gamma_k = -1$ for each $k = 0, \dots, n-1$. We denote this algebra by~$\Main_n$.
\end{definition}

\begin{proposition}[{\cite[Proposition~3.2]{biss}}] \label{proposition:A_n-euclidean-product}
If $\A_n = \Main_n$, then the bilinear form $\langle a, b \rangle$ is a Euclidean inner product. In particular, the norm $n(a)$ is anisotropic.
\end{proposition}

\begin{corollary}
The statements~\ref{lemma:moreno-tilde}--\ref{theorem:double-hexagon} are valid for all real Cayley--Dickson algebras of the main sequence. In this case the double hexagon can be simplified, and its new form is depicted in Figure~\ref{figure:double-hexagon}(b). In~\cite[Theorem~4.11]{our_orthographs1} the multiplication table of the vertices of this double hexagon has also been obtained, and it has a convenient block structure.
\end{corollary}

\begin{example}
The complex numbers ($\mathbb{C}$), the quaternions ($\mathbb{H}$), the octonions ($\mathbb{O}$), and the sedenions ($\mathbb{S}$) are the real algebras of the main sequence for $n = 1$, $2$, $3$, and $4$, correspondingly, cf.~\cite{baez}.
\end{example}

\subsection{Octonion and sedenion algebras} \label{subsection:octonions-sedenions}

Exact definitions and some basic properties of~$\mathbb{O}$ and~$\mathbb{S}$ are given below.

\begin{definition}[{\cite[p.~6]{baez}}]
{\em The octonions}~$\mathbb{O}$ are an eight-dimensional algebra over $\mathbb{R}$, its basis elements being equal to $1,e_1,\dots,e_7$. The involution in~$\mathbb{O}$ is given by the formula
$$
\overline{a_0 + a_1e_1 + \dots + a_7e_7} = a_0 - a_1e_1 - \dots - a_7e_7,
$$
and multiplication is given by Table~\ref{table:octonions-mult}.

\begin{table}[H]
\centering
$
\begin{array}{|c|cccccccc|}
\hline
\times&1&e_1&e_2&e_3&e_4&e_5&e_6&e_7\\\hline
1&1&e_1&e_2&e_3&e_4&e_5&e_6&e_7\\
e_1&e_1&-1& e_3& -e_2& e_5& -e_4& -e_7& e_6\\
e_2&e_2&-e_3& -1& e_1& e_6& e_7& -e_4& -e_5\\
e_3&e_3&e_2& -e_1& -1& e_7& -e_6& e_5& -e_4\\
e_4&e_4&-e_5& -e_6& -e_7& -1& e_1& e_2& e_3\\
e_5&e_5&e_4& -e_7& e_6& -e_1& -1& -e_3& e_2\\
e_6&e_6&e_7& e_4& -e_5& -e_2& e_3& -1& -e_1\\
e_7&e_7&-e_6& e_5& e_4& -e_3& -e_2& e_1& -1\\\hline
\end{array}
$
\caption{\label{table:octonions-mult} Multiplication table of the unit octonions.}
\end{table}
\end{definition}

It is well known that $\mathbb{O} \cong \Main_3$, and thus~$\mathbb{O}$ is a noncommutative nonassociative but alternative algebra without zero divisors, cf.~\cite[p.~10]{baez}. The elements $1, e_1, \dots, e_7$ form an orthonormal basis with respect to the inner product $\langle \cdot, \cdot \rangle$.

The algebra of {\em the sedenions} is defined as $\mathbb{S} = \Main_4$. One can easily see that $\mathbb{S} = \Main_4 = \Main_3 \{ -1 \} \cong \mathbb{O} \{ -1 \}$, and this isomorphism provides the most convenient representation for the sedenions. The sedenions are a noncommutative nonassociative and nonalternative algebra with zero divisors, cf.~\cite[p.~2]{moreno}.

\medskip

It is widely known that complex numbers can be represented by real matrices of order~$2$, and quaternions can be represented either by real matrices of order~$4$ or by complex matrices of order~$2$, as follows:
\begin{align*}
    \mathbb{C} \ni a + bi &\mapsto 
    \begin{pmatrix}
        a & -b\\
        b & a
    \end{pmatrix},\\
    \mathbb{H} \ni a + b i + c j + d k &\mapsto 
    \begin{pmatrix}
        a & -b & -c & -d\\
        b & a & -d & c\\
        c & d & a & -b\\
        d & -c & b & a
    \end{pmatrix},\\
    \mathbb{H} \ni (a + b i) + (c + di) j = x + y j &\mapsto 
    \begin{pmatrix}
        x & y\\
        -\bar{y} & \bar{x}
    \end{pmatrix} = 
    \begin{pmatrix}
        a+bi & c+di\\
        -c+di & a-bi
    \end{pmatrix}.
\end{align*}
However, since the octonions are nonassociative, they cannot be represented in a similar way. In~\cite{tian}, given an arbitrary $x \in \mathbb{O}$, its left and right eight-dimensional real matrix representations were considered, denoted by $\omega(x)$ and $\nu(x)$, respectively. They satisfy the equalities $\overrightarrow{xy} = \omega(x) \overrightarrow{y}$ and $\overrightarrow{yx} = \nu(x) \overrightarrow{y}$, where $\overrightarrow{y} \in \mathbb{R}^8$ is the coordinate vector of an element $y \in \mathbb{O}$. One of their possible applications is to solve linear and matrix equations of certain types over~$\mathbb{O}$.

Another possible way to rewrite the octonions in matrix form uses Zorn vector-matrix algebra. Over an arbitrary field~$\mathbb{F}$, it is defined as
$$
\Zorn(\mathbb{F}) = \left\{
\begin{pmatrix}
a&\mathbf{v}\\
\mathbf{w}&b\\
\end{pmatrix} \; \Big| \; a, b \in \mathbb{F}, \; \mathbf{v}, \mathbf{w} \in \mathbb{F}^3 \right\},
$$
while addition and multiplication are given by
\begin{align*}
\begin{pmatrix}
a&\mathbf{v}\\
\mathbf{w}&b\\
\end{pmatrix}
+
\begin{pmatrix}
a'&\mathbf{v'}\\
\mathbf{w'}&b'\\
\end{pmatrix}
&=
\begin{pmatrix}
a + a' & \mathbf{v} + \mathbf{v'}\\
\mathbf{w} + \mathbf{w'} & b + b'\\
\end{pmatrix},\\
\begin{pmatrix}
a&\mathbf{v}\\
\mathbf{w}&b\\
\end{pmatrix}
\begin{pmatrix}
a'&\mathbf{v'}\\
\mathbf{w'}&b'\\
\end{pmatrix}
&=
\begin{pmatrix}
a a' + \mathbf{v} \cdot \mathbf{w'} & a \mathbf{v'} + b' \mathbf{v} + \mathbf{w} \times \mathbf{w'}\\
a' \mathbf{w} + b \mathbf{w'} - \mathbf{v} \times \mathbf{v'} & b b' + \mathbf{v'} \cdot \mathbf{w}\\
\end{pmatrix},
\end{align*}
where $\cdot$ and $\times$ denote the dot product and the cross product of the elements of $\mathbb{F}^3$, respectively. Involution and norm are also defined on $\Zorn(\mathbb{F})$: 
$$
\text{if } A = \begin{pmatrix}
a&\mathbf{v}\\
\mathbf{w}&b\\
\end{pmatrix}, \text{ then } 
\bar{A} = 
\begin{pmatrix}
b&-\mathbf{v}\\
-\mathbf{w}&a\\
\end{pmatrix}
\text{ and }
n(A) = ab - \mathbf{v} \cdot \mathbf{w}.
$$
If $\A_3$ is a Cayley--Dickson algebra over an arbitrary field~$\mathbb{F}$, $\chrs \mathbb{F} \neq 2$, such that the norm on~$\A_3$ is isotropic (i.e., there exists $a \in \A_3 \setminus \{ 0 \}$ such that $n(a) = 0$), then $\A_3$ is isomorphic to $\Zorn(\mathbb{F})$, see~\cite[Theorem~1.6]{Elduque4} and~\cite[pp.~160 and~166]{mccrimmon}.

The real octonions~$\mathbb{O}$ can be considered as a real subalgebra of the complex octonions $\mathbb{O}_{\mathbb{C}} = \mathbb{C} \otimes_{\mathbb{R}} \mathbb{O}$. The norm on $\mathbb{O}_{\mathbb{C}}$ is isotropic, so there exists an isomorphism $\mathbb{O}_{\mathbb{C}} \cong \Zorn(\mathbb{C})$ which is given explicitly by the formula
$$
\mathbb{O}_{\mathbb{C}} \ni (a+c) + (b+d)e_4 \mapsto \begin{pmatrix}
a-ib & \mathbf{c}+i\mathbf{d}\\
-\mathbf{c}+i\mathbf{d} & a+ib\\
\end{pmatrix},
$$
where $a,b \in \mathbb{C}$ and $c, d \in \spn(e_1,e_2,e_3)$, with $\mathbf{c}, \mathbf{d} \in \mathbb{C}^3$ being their coordinate vectors. Indeed, one can easily verify that such a mapping preserves the multiplication table~\ref{table:octonions-mult} of the unit octonions. Besides, this isomorphism preserves involution, hence it also preserves the norm. We can restrict it to the case when $a,b \in \mathbb{R}$ and $\mathbf{c}, \mathbf{d} \in \mathbb{R}^3$, and thus obtain a representation of~$\mathbb{O}$ as a real subalgebra of~$\Zorn(\mathbb{C})$. Note that this mapping bears an extreme resemblance to the representation of quaternions by complex matrices of order~$2$.

\section{Commutativity graph of the sedenions} \label{section:commutativity-graph}

We now study the commutativity graph of the sedenions. We will need several facts on the center and centralizers which are valid for arbitrary Cayley--Dickson algebras. Note that, by~\cite[Proposition~8.19]{our_anticomm}, the condition $n \leq 3$ is essential in Lemma~\ref{lemma:A_n-commutativity-through-orthogonality}(1).

\begin{lemma}[{\cite[p.~438]{schafer}}]
\leavevmode
\begin{enumerate}[{\rm (1)}]
    \item If $n \leq 1$, then $C_{\A_n} = \A_n$, so the vertex set of $\Gamma_C(\A_n)$ is the empty set.
    \item If $n \geq 2$, then $C_{\A_n} = \mathbb{F}$, so the vertex set of $\Gamma_C(\A_n)$ is $\mathbb{P}(\A_n / \mathbb{F})$.
\end{enumerate}
\end{lemma}

\begin{lemma}[{\cite[Lemma~8.11]{our_anticomm}}] \label{lemma:A_n-commutativity-through-orthogonality}
Let $a \in \A_n \setminus \{ 0 \}$, $t(a) = 0$. Then
\begin{enumerate}[{\rm (1)}]
\item if $n(a) = 0$ and $n \leq 3$, then $C_{\A_n}(a) = \mathbb{F} \oplus O_{\A_n}(a)$;
\item if $n(a) \neq 0$, then $C_{\A_n}(a) = \mathbb{F} \oplus \mathbb{F}a \oplus O_{\A_n}(a)$.
\end{enumerate}
\end{lemma}

\begin{remark}
If the norm on~$\A_n$ is anisotropic (in particular, if $\A_n = \Main_n$), then any element $a \in \A_n \setminus \{ 0 \}$, $t(a) = 0$, satisfies the conditions of Lemma~\ref{lemma:A_n-commutativity-through-orthogonality}(2).
\end{remark}

\begin{corollary} \label{corollary:small-commutativity-component}
Let $n \geq 2$,~$\A_n$ be a Cayley--Dickson algebra with anisotropic norm, $a \in \A_n \setminus \mathbb{F}$, $\Im(a) \notin Z(\A_n)$. Then the connected component of $\Gamma_C(\A_n)$ which contains~$a$ is a singleton.
\end{corollary}

\begin{proof}
It follows from Lemma~\ref{lemma:A_n-commutativity-through-orthogonality} that $C_{\A_n}(a) = C_{\A_n}(\Im(a)) = \spn(1, \Im(a)) = \spn(1,a) = \mathbb{F}a + C_{\A_n}$. Hence the only vertex of the connected component under consideration is $[a + C_{\A_n}] = \mathbb{F}a + C_{\A_n}$.
\end{proof}

Consider the case when $\A_n = \mathbb{S}$. By Lemma~\ref{lemma:A_n-alternative-properties}(1), for any $a \in Z(\mathbb{S})$ we have $t(a) = 0$, so $a = \Im(a)$. In view of Corollary~\ref{corollary:small-commutativity-component}, it is natural to give the following definition:

\begin{definition}
$\Gamma_C^Z(\mathbb{S})$ is the subgraph of $\Gamma_C(\mathbb{S})$ on the vertex set 
$$
\mathbb{P}(Z(\mathbb{S}) + \mathbb{R}) = \{ \mathbb{R}a + \mathbb{R} \; | \; a \in Z(\mathbb{S}) \}.
$$
\end{definition}

Since there exists a simple criterion for an arbitrary sedenion to be a zero divisor, the vertex set of $\Gamma_C^Z(\mathbb{S})$ can be described explicitly.

\begin{proposition}[{\cite[Proposition~12.1]{biss}}] \label{lemma:sedenions-zero-divisor-pairs}
Let $a, b \in \mathbb{O}$, $(a,b) \in \mathbb{S} \setminus \{ 0 \}$. Then $(a,b) \in Z(\mathbb{S})$ if and only if the following conditions are satisfied:
\begin{enumerate}[{\rm (1)}]
\item $n(a) = n(b)$;
\item $1,a,b$ are orthogonal with respect to the inner product $\langle \cdot, \cdot \rangle$.
\end{enumerate}
\end{proposition}

We now show that $\Gamma_C^Z(\mathbb{S})$ is a connected component of $\Gamma_C(\mathbb{S})$ whose diameter equals three. It was proved earlier (see \cite[Proposition~6.6]{our_sedenions}) that $\Gamma_C^Z(\mathbb{S})$ is connected, and its diameter does not exceed four.

\begin{proposition} \label{proposition:lower-bound}
The diameter of $\Gamma_C^Z(\mathbb{S})$ is at least $3$.
\end{proposition}

\begin{proof}
Consider $(a,b),(c,d) \in Z(\mathbb{S})$, $(a,b)(c,d) = 0$. Then $(a,b),(c,d)$ satisfy the conditions of Theorem~\ref{theorem:double-hexagon}, and thus they are contained in a double hexagon from Figure~\ref{figure:double-hexagon}(b), all of whose vertices are linearly independent. By~\cite[p.~25]{moreno}, the orthogonalizer of an arbitrary zero divisor in~$\mathbb{S}$ has dimension~$4$. It follows that
\begin{align*}
O_{\mathbb{S}}((a,b)) &= \spn((c,d), (d,-c), (ad,ac), (ac,-ad)),\\
O_{\mathbb{S}}((b,a)) &= \spn((d,c), (c,-d), (ac,ad), (ad,-ac)).
\end{align*}
By Lemma~\ref{lemma:A_n-commutativity-through-orthogonality}(2),
\begin{align*}
C_{\mathbb{S}}((a,b)) &= \spn(1,(a,b)) \oplus O_{\mathbb{S}}((a,b)) = \spn(1, (a,b), (c,d), (d,-c), (ad,ac), (ac,-ad)),\\
C_{\mathbb{S}}((b,a)) &= \spn(1,(b,a)) \oplus O_{\mathbb{S}}((b,a)) = \spn(1, (b,a), (d,c), (c,-d), (ac,ad), (ad,-ac)).
\end{align*}
Then, by Theorem~\ref{theorem:double-hexagon}\eqref{item:linear-independence}, $C_{\mathbb{S}}((a,b)) \cap C_{\mathbb{S}}((b,a)) = \mathbb{R}$, so $d_{\Gamma_C(\mathbb{S})}((a,b),(b,a)) \geq 3$.
\end{proof}

Since~$\mathbb{O}$ is alternative, any pair of zero divisors $(a,b),(c,d) \in \mathbb{S}$, $(a,b)(c,d) = 0$, satisfies the conditions of Lemma~\ref{lemma:A_n-alternative-properties}. In particular, we can multiply $(a,b)$ and $(c,d)$ by certain nonzero real numbers and achieve that $n(a) = n(b) = n(c) = n(d) = 1$. Then, by~\cite[Corollary~2.12 and Theorem~2.13]{moreno}, there exists an automorphism $\phi$ of~$\mathbb{O}$ which maps $a,b,c,d$ to $e_1,e_2,e_7,e_4$, respectively. It can be extended to an automorphism of~$\mathbb{S}$ by the formula $(x,y) \mapsto (\phi(x),\phi(y))$. Then $(a,b)$ and $(c,d)$ are mapped to $(e_1,e_2)$ and $(e_7,e_4)$. Hence we can always replace $(a,b)$ and $(c,d)$ with $(e_1,e_2)$ and $(e_7,e_4)$. We immediately obtain the following proposition.

\begin{proposition} \label{proposition:sedenions-orthonormal-system}
Let $(a,b),(c,d) \in Z(\mathbb{S})$, $(a,b)(c,d) = 0$, $n(a) = n(b) = n(c) = n(d) = 1$. Then $1,a,b,c,d,ab,ac,ad$ form an orthonormal system with respect to the inner product $\langle \cdot, \cdot \rangle$.
\end{proposition}

By using this proposition and Corollary~\ref{lemma:double-alternative-annihilators}, one can easily find an explicit form of the orthogonalizer for an arbitrary zero divisor in~$\mathbb{S}$. It has already been obtained by Biss, Dugger, and Isaksen~\cite{biss}.

\begin{lemma}[{\cite[Corollary~11.2]{biss}}] \label{lemma:sedenions-annihilators}
Let $(a,b) \in Z(\mathbb{S})$. Then
$$
O_{\mathbb{S}}((a,b)) = \left\{ \left(c, -\dfrac{(ab)c}{n(a)} \right) \; \bigg| \; c \perp 1, a, b, ab \right\}.
$$
\end{lemma}

We will also need the following lemma which is valid in an arbitrary Cayley--Dickson algebra.

\begin{lemma}[{\cite[Lemma~6]{schafer}}] \label{lemma:inner-product-movement}
For all $x,y,z \in \A_n$ we have $\langle x, yz \rangle = \langle x\bar{z}, y \rangle = \langle \bar{y}x, z \rangle$.
\end{lemma}

\begin{corollary} \label{corollary:pull-out-norm}
Let $x,y,z \in \A_n$, and the element $z$ alternates with $x$. Then $\langle xz, yz \rangle = \langle zx, zy \rangle = n(z) \langle x, y \rangle$. 
\end{corollary}

\begin{proof}
By Lemma~\ref{lemma:inner-product-movement}, 
\begin{align*}
\langle xz, yz \rangle &= \langle (xz)\bar{z}, y \rangle = \langle (xz)(t(z) - z), y \rangle \\
&= \langle x(z(t(z) - z)), y \rangle = \langle x(z\bar{z}), y \rangle \\
&= \langle x \cdot n(z), y \rangle = n(z) \langle x, y \rangle.    
\end{align*}
The second equality is proved similarly.
\end{proof}

\begin{lemma} \label{lemma:length-two-path}
Let $(a,b), (a', b') \in Z(\mathbb{S})$. Then the distance between $(a,b)$ and $(a',b')$ in $\Gamma_C(\mathbb{S})$ is at most two
if and only if the following three equalities hold:
\begin{equation} \label{equation:length-two-path}
\begin{cases}
\langle a, a' \rangle - \langle b, b' \rangle = 0,\\
\langle a, b' \rangle + \langle b, a' \rangle = 0,\\
\begin{vmatrix}
\langle a, a'b' \rangle & \langle a', ab \rangle\\
\langle b, a'b' \rangle & \langle b', ab \rangle
\end{vmatrix} = 0.
\end{cases}
\end{equation}
\end{lemma}

\begin{proof}
We may assume without loss of generality that $n(a) = n(b) = n(a') = n(b') = 1$. By definition, $d_{\Gamma_C(\mathbb{S})}((a,b),(a',b')) \leq 2$ if and only if there exists a path of length at most two between $(a,b)$ and $(a',b')$ in $\Gamma_C(\mathbb{S})$. Clearly, it is sufficient to consider only those elements between $(a,b)$ and $(a',b')$ which are nonzero but have zero trace. By Lemma~\ref{lemma:A_n-commutativity-through-orthogonality}, the element $(a,b)$ is adjacent only to elements of the form $\alpha (a,b) + (c,d)$, where $\alpha \in \mathbb{R}$ and $(c,d) \in O_{\mathbb{S}}((a,b))$. It follows from Lemma~\ref{lemma:sedenions-annihilators} that the last condition is equivalent to $c \in \spn(1, a, b, ab)^{\perp}$ and $d = -(ab)c$. Similarly, $(a',b')$ is adjacent to elements of the form $\beta (a',b') + (c',d')$, where $\beta \in \mathbb{R}$, $c' \in \spn(1, a', b', a'b')^{\perp}$ and $d' = -(a'b')c'$. 

We need the equality $\alpha (a,b) + (c,d) = \beta (a',b') + (c',d')$ to be satisfied, which is equivalent to the following system:
\begin{equation} \label{equation:equal-vertices}
\begin{cases}
    c' = c + \alpha a - \beta a',\\
    d' = d + \alpha b - \beta b'.
\end{cases}
\end{equation}
Then $(a,b) \leftrightarrow \alpha (a,b) + (c,d) = \beta (a',b') + (c',d') \leftrightarrow (a',b')$ is the desired path of length at most two. Note that, if its middle element is proportional either to $(a,b)$ or to $(a',b')$, then this path has two subsequent equal vertices, but $\Gamma_C(\mathbb{S})$ has no loops, so it is actually a path of length at most one. Substituting expressions for $d$ and $d'$ into system~\eqref{equation:equal-vertices}, we obtain
\begin{equation} \label{equation:condition-on-c}
\begin{aligned}
    -(ab)c + \alpha b - \beta b' = -(a'b')c' &= -(a'b')(c + \alpha a - \beta a') \\
    &= -(a'b')c - \alpha(a'b')a + \beta (a'b')a'.
\end{aligned}
\end{equation}
By Lemma~\ref{lemma:A_n-anticomm} and Proposition~\ref{proposition:sedenions-orthonormal-system}, the elements $a'$ and $b'$ anticommute, hence $(a'b')a' = -(b'a')a' = -b'(a')^2 = n(a')b' = b'$. We rearrange Eq.~\eqref{equation:condition-on-c}, so that all summands on the left-hand side contain the element $c$:
\begin{equation} \label{equation:condition-on-c-2}
(a'b' - ab)c = 2\beta b' - \alpha b - \alpha (a'b')a.
\end{equation}

Consider now two cases. If $ab = a'b'$, then we set $\alpha = \beta = 0$ and choose an arbitrary nonzero $c = c' \in \spn(1, a, b, a', b', ab)^{\perp}$. Then $(a,b) \leftrightarrow (c, -(ab)c) = (c, -(a'b')c) \leftrightarrow (a',b')$ is the desired path of length two. By Proposition~\ref{proposition:sedenions-orthonormal-system}, $\langle a, ab \rangle = \langle b, ab \rangle = \langle a', a'b' \rangle = \langle b', a'b' \rangle = 0$, so in this case the determinant in system~\eqref{equation:length-two-path} is equal to zero. Besides, by Corollary~\ref{corollary:pull-out-norm}, 
\begin{align*}
\langle a, a' \rangle - \langle b, b' \rangle &= \langle a, a' \rangle - \langle a'b, a'b' \rangle \\
&= \langle a, a' \rangle - \langle a'b, ab \rangle \\
&= \langle a, a' \rangle - n(b) \langle a', a \rangle = 0,\\
\langle a, b' \rangle + \langle b, a' \rangle &= \langle a, b' \rangle + \langle bb', a'b' \rangle \\
&= \langle a, b' \rangle + \langle bb', ab \rangle \\
&= \langle a, b' \rangle - \langle bb', ba \rangle \\
&= \langle a, b' \rangle - n(b) \langle b', a \rangle = 0.
\end{align*}
Therefore, in this case all equations of system~\eqref{equation:length-two-path} are satisfied, and the statement of the lemma holds true.

We assume further that $ab \neq a'b'$. Since~$\mathbb{O}$ is a division algebra, the element~$c$ is uniquely determined by the parameters $\alpha, \beta \in \mathbb{R}$ from Eq.~\eqref{equation:condition-on-c-2}. Clearly, $\alpha = \beta = 0$ implies that $c = 0$, so $\alpha (a,b) + (c,d) = (0,0)$. Thus we have to determine the conditions under which there exists a nonzero pair of parameters $(\alpha, \beta)$ such that $c \in \spn(1, a, b, ab)^{\perp}$ and $c' \in \spn(1, a', b', a'b')^{\perp}$.

Note that $n(a'b' - ab) \neq 0$, so $\langle c, x \rangle = 0$ if and only if $n(a'b' - ab) \langle c, x \rangle = 0$. By Corollary~\ref{corollary:pull-out-norm}, 
$$
n(a'b' - ab) \langle c, x \rangle = \langle (a'b' - ab)c, (a'b' - ab)x \rangle = \langle 2\beta b' - \alpha b - \alpha (a'b')a, (a'b' - ab)x \rangle.
$$
Hence we have $c \in \spn(1, a, b, ab)^{\perp}$ if and only if the equality $\langle 2\beta b' - \alpha b - \alpha (a'b')a, (a'b' - ab)x \rangle = 0$ holds for all $x \in \{ 1, a, b, ab \}$. 

Our further computations are based on Lemma~\ref{lemma:inner-product-movement} and Corollary~\ref{corollary:pull-out-norm}. Besides, we will need the fact that, by Proposition~\ref{proposition:sedenions-orthonormal-system}, the elements $1, a, b, ab$ (and, similarly, $1, a', b', a'b'$) form an orthonormal system with respect to the inner product $\langle \cdot, \cdot \rangle$. It then follows from Lemma~\ref{lemma:A_n-anticomm} that $a, b, ab$ (and $a', b', a'b'$) anticommute pairwise.

\begin{itemize}
    \item Let $x = 1$. Note that
    $$
    \langle (a'b')a, ab \rangle = - \langle (a'b')a, ba \rangle = -n(a) \langle a'b', b \rangle = - \langle b, a'b' \rangle.
    $$
    Then
    \begin{align*}
    0 &{} = \langle 2\beta b' - \alpha b - \alpha (a'b')a, (a'b' - ab)x \rangle \\
    &{} = \langle 2\beta b' - \alpha b - \alpha (a'b')a, a'b' - ab \rangle \\
    &{} = - 2\beta \langle b', ab \rangle - 2\alpha \langle b, a'b' \rangle - \alpha n(a'b') \langle a, 1 \rangle \\
    &{} = \mathrlap{-2(\alpha \langle b, a'b' \rangle + \beta \langle b', ab \rangle).} \phantom{2\beta \langle b'a, b'a' \rangle - 2\beta \langle b', b \rangle + \alpha \langle ba, a'b' \rangle - \alpha \langle a'b', ba \rangle}
    \end{align*}
    \item Let $x = a$. Then
    \begin{align*}
    0 &{} = \langle 2\beta b' - \alpha b - \alpha (a'b')a, (a'b' - ab)x \rangle \\
    &{} = \langle 2\beta b' - \alpha b - \alpha (a'b')a, (a'b')a - b \rangle \\
    &{} = 2\beta \langle b', (a'b')a \rangle - 2\beta \langle b', b \rangle - \alpha \langle b, (a'b')a \rangle \\
    &{} + \alpha n(b) - \alpha n(a'b') n(a) + \alpha \langle (a'b')a, b \rangle \\
    &{} = 2\beta \langle b'a, b'a' \rangle - 2\beta \langle b', b \rangle + \alpha \langle ba, a'b' \rangle - \alpha \langle a'b', ba \rangle \\
    &{} = 2\beta (\langle a, a' \rangle - \langle b, b' \rangle).
    \end{align*}
    \item Let $x = b$. Then
    \begin{align*}
    0 &{} = \langle 2\beta b' - \alpha b - \alpha (a'b')a, (a'b' - ab)x \rangle \\
    &{} = \langle 2\beta b' - \alpha b - \alpha (a'b')a, (a'b')b + a \rangle \\
    &{} = 2\beta \langle b', (a'b')b \rangle + 2\beta \langle b', a \rangle - \alpha n(b) \langle 1, a'b' \rangle \\
    &{} - \alpha n(a'b') \langle a, b \rangle - \alpha n(a) \langle a'b',1 \rangle \\
    &{} = 2\beta \langle b'b, b'a' \rangle + 2\beta \langle b', a \rangle \\
    &{} = \mathrlap{2\beta (\langle a', b \rangle + \langle a, b' \rangle).} \phantom{2\beta \langle b'a, b'a' \rangle - 2\beta \langle b', b \rangle + \alpha \langle ba, a'b' \rangle - \alpha \langle a'b', ba \rangle}
    \end{align*}
    \item Let $x = ab$. Then
    \begin{align*}
    0 &{} = \langle 2\beta b' - \alpha b - \alpha (a'b')a, (a'b' - ab)x \rangle \\
    &{} = \langle 2\beta b' - \alpha b - \alpha (a'b')a, (a'b')(ab) + 1 \rangle \\
    &{} = 2\beta \langle b', (a'b')(ab) \rangle - \alpha \langle b, (a'b')(ab) \rangle \\
    &{} - \alpha n(a'b') \langle a, ab \rangle - \alpha \langle (a'b')a, 1 \rangle\\
    &{} = 2\beta \langle b'(ab), b'a' \rangle + \alpha \langle (a'b')b, ab \rangle + \alpha \langle a'b', a \rangle \\
    &{} = \mathrlap{2(\alpha \langle a, a'b' \rangle + \beta \langle a', ab \rangle).} \phantom{2\beta \langle b'a, b'a' \rangle - 2\beta \langle b', b \rangle + \alpha \langle ba, a'b' \rangle - \alpha \langle a'b', ba \rangle}
    \end{align*}
\end{itemize}

Since the roles of $(a,b), (c,d)$ and $(a',b'), (c',d')$ are interchangeable, the condition $c' \in \spn(1, a', b', a'b')^{\perp}$ can be obtained from the above equations by substituting $\alpha \leftrightarrow \beta$, $a \leftrightarrow a'$, and $b \leftrightarrow b'$. In total we have the following six equalities:
$$
\begin{cases}
\alpha (\langle a, a' \rangle - \langle b, b' \rangle) = 0,\\
\beta (\langle a, a' \rangle - \langle b, b' \rangle) = 0,\\
\alpha (\langle a', b \rangle + \langle a, b' \rangle) = 0,\\
\beta (\langle a', b \rangle + \langle a, b' \rangle) = 0,\\
\alpha \langle b, a'b' \rangle + \beta \langle b', ab \rangle = 0,\\
\alpha \langle a, a'b' \rangle + \beta \langle a', ab \rangle = 0.
\end{cases}
$$
Since at least one of the parameters $\alpha, \beta$ is nonzero, the first four equations are equivalent to $\langle a, a' \rangle - \langle b, b' \rangle = \langle a', b \rangle + \langle a, b' \rangle = 0$. Then there exists a nonzero pair $(\alpha, \beta)$ which satisfies the last two equations if and only if the determinant of their coefficient matrix is equal to zero.
\end{proof}

\begin{theorem} \label{theorem:diameter-3}
The diameter of $\Gamma_C^Z(\mathbb{S})$ equals $3$.    
\end{theorem}

\begin{proof}
By Proposition~\ref{proposition:lower-bound}, the diameter of $\Gamma_C^Z(\mathbb{S})$ is at least three. Thus it is sufficient to show that there exists a path of length at most three in $\Gamma_C(\mathbb{S})$ between any two elements $(a,b), (c,d) \in Z(\mathbb{S})$. To prove this, we will find an element $(a',b') \in Z(\mathbb{S})$ such that
\begin{enumerate}[{\rm (1)}]
    \item there exists a path of length at most two between $(a,b)$ and $(a',b')$,
    \item $(a',b')$ and $(c,d)$ are either linearly dependent 
    or adjacent, i.e., $(a',b') \in C_{\mathbb{S}}((c,d))$.
\end{enumerate}

Note that, if $\langle a', ab \rangle = \langle b', ab \rangle = 0$, then the determinant in Lemma~\ref{lemma:length-two-path} is equal to zero. Thus system~\eqref{equation:length-two-path} holds automatically if the following system of equations is satisfied:
\begin{equation} \label{equation:length-two-path-simple}
\begin{cases}
\langle a, a' \rangle - \langle b, b' \rangle = 0,\\
\langle a, b' \rangle + \langle b, a' \rangle = 0,\\
\langle a', ab \rangle = 0,\\
\langle b', ab \rangle = 0.
\end{cases}
\end{equation}
It follows from Proposition~\ref{proposition:real-cayley-dickson-properties} that $\langle (x,y), (z,w) \rangle = \langle x, z \rangle + \langle y, w \rangle$ for all $(x,y), (z,w) \in \mathbb{S}$. Thus system~\eqref{equation:length-two-path-simple} is satisfied if and only if
$$
(a',b') \in \spn((a,-b), (b,a), (ab,0), (0,ab))^{\perp}.
$$

Since $(a',b') \in Z(\mathbb{S})$, we have $t((a',b')) = 0$. Hence, by Lemma~\ref{lemma:A_n-commutativity-through-orthogonality}(2), condition~(2) is equivalent to the fact that $(a',b') \in \Im(C_{\mathbb{S}}((c,d))) = \mathbb{R}(c,d) \oplus O_{\mathbb{S}}((c,d))$. But $\dim \Im(C_{\mathbb{S}}((c,d))) = 5$, so there exists a nonzero element
$$
(a',b') \in \Im(C_{\mathbb{S}}((c,d))) \cap \spn((a,-b), (b,a), (ab,0), (0,ab))^{\perp}.
$$
This element is the desired one.
\end{proof}

\medskip

The author is grateful to her scientific advisor Professor Alexander E. Guterman for posing the
problem and fruitful discussions.

\end{document}